\documentclass[11pt]{article}
\usepackage[utf8]{inputenc}
\usepackage{amsmath,amssymb,amsthm}
\usepackage[numbers,sort&compress]{natbib}
\usepackage{booktabs}
\usepackage{array}
\usepackage{tabularx}
\usepackage{geometry}
\usepackage{caption}
\usepackage{float}
\usepackage{hyperref}
\usepackage{authblk}
\usepackage{orcidlink}
\usepackage{xcolor}
\usepackage{tikz}
\newtheorem{theorem}{Theorem}
\newtheorem{lemma}[theorem]{Lemma}
\newtheorem{proposition}[theorem]{Proposition}
\newtheorem{corollary}[theorem]{Corollary}
\newcommand{\zl}[1]{Z^+_{(1)}(#1)}
\newcommand{\zleak}[2]{Z^+_{(#1)}(#2)}
\newcommand{\sbd}[2]{\partial^{#1}_1(#2)}
\newcommand{\set}[1]{\{#1\}}
\newcolumntype{Y}{>{\raggedright\arraybackslash}X}
\hypersetup{
  colorlinks=true,
  linkcolor=blue!55!black,
  citecolor=blue!55!black,
  urlcolor=blue!55!black,
  pdftitle={Sharp Edge-Edit Bounds at Every Level for Leaky Positive Semidefinite Forcing},
  pdfauthor={Domenico Frijio}
}

\begin{document}

\title{\LARGE\textbf{Sharp Edge-Edit Bounds at Every Level for Leaky Positive Semidefinite Forcing}}
\author[1]{Domenico Frijio\orcidlink{0009-0005-2747-3961}}
\affil[1]{Horizon Research\\
  {\fontsize{9}{9}\selectfont\texttt{dfrijio@horizonrsc.com}}}
\date{19 September 2026}
\maketitle

\begin{abstract}
This work disproves the 1-leaky positive semidefinite edge-deletion conjecture and replaces it with a sharp theorem.
For every leak level \(\ell\) and edge \(e\), one has \(\lvert\zleak{\ell}{G}-\zleak{\ell}{G-e}\rvert\le2\).
More generally, if two graphs differ only on edges with both endpoints in \(S\), their parameters differ by at most \(\lvert S\rvert\).
An endpoint-sensitive refinement recovers an increase of at most one whenever some minimum set for \(G-e\) contains an endpoint of \(e\).
Both signs are sharp for every positive leak level.
Joining two copies of \(K_{\ell+1}\) by a bridge gives \(\zleak{\ell}{G}=2\ell\) and \(\zleak{\ell}{G-e}=2\ell+2\).
For every \(\ell\ge2\), a connected clique-leaf pair of order \(2\ell+3\) gives the opposite difference.
The remaining positive-difference one-leak case is attained by connected graphs on nine vertices with \(\zl H=4\) and \(\zl G=6\).
Explicit forcing sequences, fort certificates, and an exact verifier check the finite extremal example and stress-test the general results.
\end{abstract}


\section{Introduction}

Leaky positive semidefinite forcing combines the PSD color-change rule with failed forcing vertices.
Zero forcing arose from inverse eigenvalue problems \cite{AIM2008,Barioli2010}, and edge spread was studied in \cite{Edholm2012}.
The PSD rule and its matrix bounds were developed in \cite{Ekstrand2013,Peters2012}.
Standard leaky forcing was introduced in \cite{Alameda2024}; resilience and generalized leak models were developed in \cite{Alameda2022,Alameda2023}.
Recent extensions treat graph products and exact families \cite{Herrman2024}.
Computational methods appear in \cite{Brimkov2019,Smith2020,Fallat2016}, while controllability gives a parallel motivation \cite{Burgarth2013}.

For standard leaky forcing, Bjorkman et al. proved the sharp bound
\(\lvert Z_{(\ell)}(G)-Z_{(\ell)}(G-e)\rvert\le2\), with both signs sharp for every \(\ell\ge1\), in their Theorem~5.2 \cite{Bjorkman2025}.
The standard clause below is an independent support-transfer proof of that result.
Its extension to PSD forcing is the issue considered here.

Elias et al. introduced \(Z^+_{(\ell)}\), proved one side of the 1-leaky edge-spread estimate, and proposed
\begin{equation}\label{eq:conjecture}
  Z^+_{(1)}(G)\le Z^+_{(1)}(G-e)+1.
\end{equation}
They also defined possibly disconnected \(\ell\)-leaky PSD forts and characterized forcing sets as their transversals \cite[Definition~32 and Theorem~34]{Elias2025}.
Theorem~\ref{thm:fort} below simplifies that characterization: only connected forts are needed, and each is specified by one singleton-boundary inequality.
Thus the novelty is the connected reduction and its explicit minimum-transversal formulation, not the general fort obstruction.
Elias et al. reported exhaustive verification of \eqref{eq:conjecture} for connected graphs through order eight \cite{Elias2025}.

We determine the optimal universal PSD replacement and prove sharpness in both directions at every positive leak level.
The main engine is a blue-support transfer theorem for arbitrary edge edits supported on a vertex set.
A connected clique-leaf family handles every \(\ell\ge2\); an order-nine pair handles one leak.

\begin{theorem}\label{thm:main}
For every finite simple graph \(G\), every \(e\in E(G)\), and every \(0\le\ell\le |V(G)|\),
\[
  -2\le \zleak{\ell}{G}-\zleak{\ell}{G-e}\le2.
\]
For every \(\ell\ge1\), both constants are attained.
The lower equality is attained with \(G\) connected.
The upper equality is attained with both graphs connected; for \(\ell\ge2\), they may be chosen on \(2\ell+3\) vertices.
\end{theorem}

\section{Universal support transfer and sharp edge spread}

Let \(X\) be a finite simple graph.
At a forcing stage, let \(B\) be the blue set and let \(C\) be a component of \(X-B\).
A blue vertex \(x\) may PSD-force \(y\in C\) when
\[
  N_X(x)\cap C=\set{y}.
\]
A leak is a vertex forbidden to perform a force.
A set is \(\ell\)-leaky PSD forcing if it forces all vertices for every set of \(\ell\) leaks.
Its minimum size is \(Z^+_{(\ell)}(X)\).
The standard rule replaces the component \(C\) by the full white set; its leaky number is \(Z_{(\ell)}(X)\).

\begin{lemma}[Fixed-leak monotonicity]\label{lem:monotone}
Fix a graph, a leak set \(L\), and either the standard or PSD rule.
If \(B\) forces the graph while no vertex of \(L\) performs a force, then every superset of \(B\) does so.
\end{lemma}

\begin{proof}
Start with a forcing sequence from \(B\).
From a larger blue set, skip each force whose target is already blue.
For every retained PSD force, the current white component containing its target is contained in the former component.
The source still has that target as its unique neighbor in the component.
The standard case is the same with one white set.
No source in \(L\) is introduced.
\end{proof}

\begin{theorem}[Blue-support transfer]\label{thm:support}
Let \(X\) and \(Y\) be graphs on the same vertex set, and let \(S\subseteq V(X)\) satisfy
\[
  E(X)\mathbin{\triangle}E(Y)\subseteq \binom{S}{2}.
\]
For either the standard or PSD rule, if \(B\) forces \(X\) with a fixed leak set \(L\), then \(B\cup S\) forces \(Y\) with the same leaks.
Consequently, for every feasible \(\ell\),
\begin{equation}\label{eq:support-bound}
  \bigl|\zleak{\ell}{X}-\zleak{\ell}{Y}\bigr|\le |S|,
\end{equation}
and the analogous bound holds for \(Z_{(\ell)}\).
\end{theorem}

\begin{proof}
Fix \(L\) and a forcing sequence from \(B\) in \(X\) avoiding \(L\).
By Lemma~\ref{lem:monotone}, start instead from \(B\cup S\) and omit forces into vertices already blue.
Every retained target lies outside \(S\).
At each stage, \(X\) and \(Y\) have the same white graph and the same edges from a blue vertex to a white vertex.
Thus every retained force is legal in \(Y\).
This proves the fixed-leak statement.
Apply it to every \(L\), choose a minimum set for \(X\), and reverse \(X,Y\) to obtain \eqref{eq:support-bound}.
\end{proof}

Taking \(S\) to be the two endpoints of one edited edge gives the following result.

\begin{corollary}[Universal edge spread]\label{cor:edge}
For every \(e\in E(G)\) and every feasible \(\ell\),
\[
  -2\le \zleak{\ell}{G}-\zleak{\ell}{G-e}\le2.
\]
The same inequality holds with \(Z_{(\ell)}\) in place of \(Z^+_{(\ell)}\).
For \(\ell=0\), the known sharp PSD bound is one \cite{Ekstrand2013}.
\end{corollary}

There is also a useful one-sided refinement.
Let \(H=G-uv\), and let \(\mathcal M^+_{\ell}(H)\) be the family of minimum \(\ell\)-leaky PSD forcing sets of \(H\).
Define
\[
  r_H(u,v)=\max_{B\in\mathcal M^+_{\ell}(H)}|B\cap\set{u,v}|.
\]

\begin{corollary}[Endpoint profile]\label{cor:endpoint}
With the notation above,
\[
  \zleak{\ell}{G}\le \zleak{\ell}{H}+2-r_H(u,v).
\]
Hence \eqref{eq:conjecture} holds whenever some minimum set of \(H\) contains an endpoint of \(uv\).
If the increase is two, every minimum set of \(H\) avoids both endpoints.
\end{corollary}

\begin{proof}
Choose \(B\in\mathcal M^+_{\ell}(H)\) attaining \(r_H(u,v)\).
Theorem~\ref{thm:support} transfers \(B\cup\set{u,v}\) to \(G\), and this set has the stated size.
\end{proof}

For the standard rule, Corollary~\ref{cor:edge} recovers Theorem~5.2 of Bjorkman et al. \cite{Bjorkman2025}.
Their proof uses standard leaky forts.
The sequence-transfer proof above also covers the PSD rule and any collection of edge edits supported on \(S\).

\section{Universal negative equality}

We first record the low-degree obstruction from standard leaky forcing \cite[Lemma~2.4]{Alameda2024}.
The same proof applies to the PSD rule.

\begin{lemma}[Low-degree obstruction]\label{lem:degree}
For either rule, every vertex of degree at most \(\ell\) belongs to every \(\ell\)-leaky forcing set.
\end{lemma}

\begin{proof}
If such a vertex \(x\) were initially white, choose an \(\ell\)-set of leaks containing \(N(x)\).
No neighbor could force \(x\).
\end{proof}

\begin{theorem}[Uniform negative equality]\label{thm:uniform}
For every integer \(\ell\ge1\), let \(H_\ell\) be the disjoint union of two copies of \(K_{\ell+1}\).
Choose one vertex in each copy, call them \(u\) and \(v\), and set \(G_\ell=H_\ell+uv\).
Then, for both the PSD and standard rules,
\[
  Z^+_{(\ell)}(G_\ell)=Z_{(\ell)}(G_\ell)=2\ell,
  \qquad
  Z^+_{(\ell)}(H_\ell)=Z_{(\ell)}(H_\ell)=2\ell+2.
\]
Thus deleting the bridge \(uv\) gives equality at \(-2\) for every positive leak level.
\end{theorem}

\begin{proof}
Every vertex of \(H_\ell\) has degree \(\ell\), so Lemma~\ref{lem:degree} forces the whole vertex set to be blue initially.

Let \(A\) and \(C\) be the vertices different from \(u\) and \(v\) in the first and second clique, respectively.
Every vertex of \(A\cup C\) has degree \(\ell\) in \(G_\ell\).
Hence every \(\ell\)-leaky forcing set has size at least \(2\ell\).
Start with \(A\cup C\) blue and fix an \(\ell\)-set \(L\) of leaks.
If neither \(A\) nor \(C\) is contained in \(L\), an unfailed vertex of each set forces its adjacent white endpoint.
If \(A\subseteq L\), then \(L=A\); any vertex of \(C\) forces \(v\), and then \(v\) forces \(u\).
The case \(C\subseteq L\) is symmetric.
These forces are legal under both rules.
\end{proof}

\section{Connected transversals and universal positive equality}

Elias et al., Definition~32 and Theorem~34, characterize leaky PSD forcing sets as transversals of a possibly disconnected fort family \cite{Elias2025}.
The next theorem gives an equivalent connected reduction.
It simplifies their formulation by discarding redundant disconnected unions, replacing the componentwise fort condition with \(\lvert\sbd XF\rvert\le\ell\), and identifying the parameter with the resulting transversal number.

For nonempty connected \(F\subseteq V(X)\), define its singleton boundary by
\[
  \sbd XF=\set{x\in V(X)\setminus F:|N_X(x)\cap F|=1}.
\]

\begin{theorem}[Connected-fort transversal]\label{thm:fort}
A set \(B\subseteq V(X)\) is \(\ell\)-leaky PSD forcing if and only if it meets every nonempty connected set \(F\) satisfying \(\lvert\sbd XF\rvert\le\ell\).
Consequently, \(Z^+_{(\ell)}(X)\) is the minimum transversal number of this family.
\end{theorem}

\begin{proof}
First suppose that \(B\cap F=\varnothing\) and choose an \(\ell\)-set of leaks containing \(\sbd XF\).
Before the first force into \(F\), every possible source belongs to \(\sbd XF\), hence is a leak.
Thus \(B\) is not \(\ell\)-leaky PSD forcing.

Conversely, fix \(\ell\) leaks and suppose the process from \(B\) stops with white vertices.
Let \(F\) be a component of the final white graph.
Every vertex in \(\sbd XF\) is blue; if it were not a leak, it could force its unique neighbor in \(F\).
Hence \(\sbd XF\) is contained in the leak set, so \(\lvert\sbd XF\rvert\le\ell\), while \(B\cap F=\varnothing\).
\end{proof}

We call the sets in Theorem~\ref{thm:fort} connected \(\ell\)-leaky PSD forts.

\begin{theorem}[Uniform positive equality]\label{thm:positive}
Fix \(\ell\ge2\), put \(a=\lfloor\ell/2\rfloor\), \(b=\lceil\ell/2\rceil\), and \(q=b+3\).
Start with a clique \(Q=K_q\) containing distinct vertices \(u,v,w\).
Attach \(a\) leaves to each of \(u\) and \(v\), and attach \(b\) leaves to \(w\).
Call the resulting graph \(X_\ell\), and set \(Y_\ell=X_\ell-uv\).
Then both graphs are connected, have order \(2\ell+3\), and satisfy
\[
  Z^+_{(\ell)}(X_\ell)=2\ell+2,
  \qquad
  Z^+_{(\ell)}(Y_\ell)=2\ell.
\]
Thus equality at \(+2\) occurs for every \(\ell\ge2\).
\end{theorem}

\begin{proof}
Let \(P_u,P_v,P_w\) be the three leaf groups, and let \(P\) be their union.
Every leaf belongs to every \(\ell\)-leaky set by Lemma~\ref{lem:degree}.

For \(X_\ell\), start from \(P\cup(Q\setminus\set{u})\).
If some leaf in \(P_u\) is not a leak, it forces \(u\).
Otherwise all \(a\) leaves in \(P_u\) are leaks, leaving at most \(b\) further leaks.
Since \(u\) has \(q-1=b+2\) blue core neighbors, one of them forces \(u\).
This gives \(Z^+_{(\ell)}(X_\ell)\le |V(X_\ell)|-1\).
If two core vertices \(x,y\) were initially white, then \(F=\set{x,y}\) would be connected and \(\lvert\sbd{X_\ell}F\rvert\le a+b=\ell\).
Theorem~\ref{thm:fort} gives the reverse inequality.

For \(Y_\ell\), start from \(P\cup(Q\setminus\set{u,v,w})\).
Call a leaf group blocked when all its leaves are leaks.
All three groups cannot be blocked because their total size is \(2a+b>\ell\).
If no group is blocked, the leaves force \(u,v,w\).
If exactly one group is blocked, the other two core vertices are forced first; the blocked vertex then has more blue core neighbors than remaining leaks.
If \(P_u,P_w\) are blocked, all \(a+b=\ell\) leaks lie there; after a leaf forces \(v\), the forces \(v\to w\to u\) finish the graph.
The case \(P_v,P_w\) is symmetric.
If \(P_u,P_v\) are blocked, at most \(b-a\le1\) leak remains; after \(w\) becomes blue, a nonleak in \(Q\setminus\set{u,v}\) forces both singleton white components.
Thus \(Z^+_{(\ell)}(Y_\ell)\le |V(Y_\ell)|-3\).

For the lower bound, suppose at least four core vertices are white.
If two unmarked core vertices \(x,y\) are white, then \(\set{x,y}\) is a connected fort with empty singleton boundary.
Otherwise all of \(u,v,w\) and one unmarked vertex \(x\) are white.
The set \(\set{x,u,v}\) is connected in \(Y_\ell\) and has singleton boundary \(P_u\cup P_v\), of size \(2a\le\ell\).
Theorem~\ref{thm:fort} completes the proof.
\end{proof}

\section{The one-leak extremal pair}

Let \(V(H)=\set{0,1,\ldots,8}\) and
\begin{align}\label{eq:edges}
E(H)=\set{&02,04,07,08,13,14,16,17,\notag\\
           &23,25,27,35,36,46,47,56}.
\end{align}
Here \(ij\) denotes the edge \(\set{i,j}\).
Set
\[
  G=H+01,\qquad e=01.
\]
Both graphs are connected.
Figure~\ref{fig:graph} displays \(G\); deleting the red edge gives \(H\).

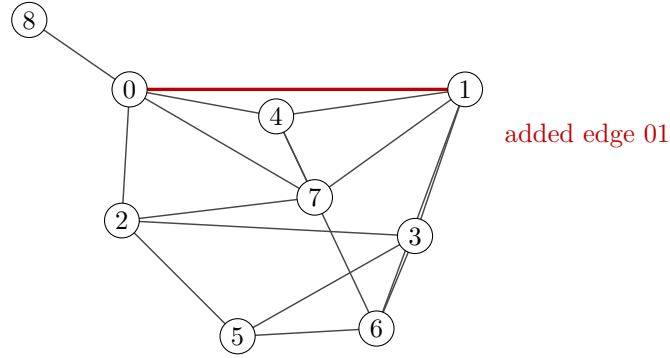
\begin{figure}[htbp]
\centering
\begin{tikzpicture}[
  scale=1.02,
  every node/.style={circle,draw,fill=white,inner sep=1.7pt,font=\small},
  baseedge/.style={draw=black!72,line width=0.55pt}
]
  \coordinate (p0) at (-2.25,1.55);
  \coordinate (p1) at ( 2.10,1.55);
  \coordinate (p2) at (-2.35,-0.15);
  \coordinate (p3) at ( 1.45,-0.35);
  \coordinate (p4) at (-0.35,1.20);
  \coordinate (p5) at (-0.85,-1.65);
  \coordinate (p6) at ( 0.95,-1.55);
  \coordinate (p7) at ( 0.15,0.15);
  \coordinate (p8) at (-3.55,2.45);
  \foreach \a/\b in {0/2,0/4,0/7,0/8,1/3,1/4,1/6,1/7,2/3,2/5,2/7,3/5,3/6,4/6,4/7,5/6}
    \draw[baseedge] (p\a)--(p\b);
  \draw[red!75!black,line width=1.25pt] (p0)--(p1);
  \foreach \v in {0,1,2,3,4,5,6,7,8}
    \node at (p\v) {\v};
  \node[draw=none,rectangle,fill=none,text=red!75!black,anchor=west] at (2.55,0.95) {added edge \(01\)};
\end{tikzpicture}
\caption{The graph \(G\). The red edge is \(e=01\).}
\label{fig:graph}
\end{figure}

\subsection{The deleted-edge graph}

\begin{proposition}\label{prop:H}
For the graph \(H\) in \eqref{eq:edges}, \(\zl H=4\).
\end{proposition}

\begin{proof}
For the upper bound, take
\[
  B_H=\set{2,3,5,8}.
\]
Table~\ref{tab:Hforces} gives a valid PSD forcing sequence for every leak \(\lambda\).
The source of each force differs from \(\lambda\).

\begin{table}[H]
\centering
\caption{Forcing certificates for \(H\).}
\label{tab:Hforces}
\begin{tabularx}{\textwidth}{@{}cY@{}}
\toprule
Leak \(\lambda\) & PSD forcing sequence \\
\midrule
\(0,4,7,8\) & \(5\to6,\ 3\to1,\ 6\to4,\ 1\to7,\ 2\to0\) \\
\(1\) & \(5\to6,\ 3\to1,\ 6\to4,\ 8\to0,\ 0\to7\) \\
\(2\) & \(5\to6,\ 3\to1,\ 6\to4,\ 1\to7,\ 4\to0\) \\
\(3\) & \(5\to6,\ 8\to0,\ 2\to7,\ 0\to4,\ 4\to1\) \\
\(5\) & \(8\to0,\ 2\to7,\ 0\to4,\ 7\to1,\ 1\to6\) \\
\(6\) & \(5\to6,\ 3\to1,\ 8\to0,\ 2\to7,\ 0\to4\) \\
\bottomrule
\end{tabularx}
\end{table}

For the lower bound, Table~\ref{tab:Hforts} lists six connected forts.
The second column gives the complete singleton boundary.

\begin{table}[H]
\centering
\caption{A fort certificate for \(H\).}
\label{tab:Hforts}
\begin{tabular}{@{}cc@{}}
\toprule
Fort \(F\) & \(\sbd HF\) \\
\midrule
\(\set{8}\) & \(\set{0}\) \\
\(\set{3,5}\) & \(\set{1}\) \\
\(\set{0,3,4,6}\) & \(\set{8}\) \\
\(\set{0,4,5,6}\) & \(\set{8}\) \\
\(\set{1,2,4,7}\) & \(\set{5}\) \\
\(\set{1,2,6,7}\) & \(\varnothing\) \\
\bottomrule
\end{tabular}
\end{table}

Every fort transversal contains \(8\).
It needs at least three further vertices to meet the remaining five forts.
Indeed, a two-vertex transversal must choose \(3\) or \(5\).
After choosing \(3\), its other vertex would lie in
\[
 \set{0,4,5,6}\cap\set{1,2,4,7}\cap\set{1,2,6,7}=\varnothing.
\]
After choosing \(5\), replace the first set by \(\set{0,3,4,6}\); the intersection is again empty.
Theorem~\ref{thm:fort} gives \(\zl H\ge4\).
The forcing table gives equality.
\end{proof}

\subsection{The restored-edge graph}

\begin{proposition}\label{prop:G}
For \(G=H+01\), \(\zl G=6\).
\end{proposition}

\begin{proof}
For the upper bound, take
\[
  B_G=B_H\cup\set{0,1}=\set{0,1,2,3,5,8}.
\]
Corollary~\ref{cor:edge} already proves that this set works.
Table~\ref{tab:Gforces} gives direct certificates.

\begin{table}[H]
\centering
\caption{Forcing certificates for \(G\).}
\label{tab:Gforces}
\begin{tabularx}{\textwidth}{@{}cY@{}}
\toprule
Leak \(\lambda\) & PSD forcing sequence \\
\midrule
\(0\) & \(2\to7,\ 3\to6,\ 1\to4\) \\
\(1\) & \(2\to7,\ 0\to4,\ 3\to6\) \\
\(2\) & \(3\to6,\ 6\to4,\ 0\to7\) \\
\(3\) & \(2\to7,\ 0\to4,\ 1\to6\) \\
\(4,5,6,7,8\) & \(2\to7,\ 0\to4,\ 1\to6\) \\
\bottomrule
\end{tabularx}
\end{table}

For the lower bound, first consider the six forts in Table~\ref{tab:Gsmall}.

\begin{table}[H]
\centering
\caption{Small forts of \(G\).}
\label{tab:Gsmall}
\begin{tabular}{@{}cc@{}}
\toprule
Fort \(F\) & \(\sbd GF\) \\
\midrule
\(\set{8}\) & \(\set{0}\) \\
\(\set{1,4}\) & \(\set{3}\) \\
\(\set{3,5}\) & \(\set{1}\) \\
\(\set{0,7}\) & \(\set{8}\) \\
\(\set{1,2,3}\) & \(\set{4}\) \\
\(\set{1,3,7}\) & \(\set{5}\) \\
\bottomrule
\end{tabular}
\end{table}

Table~\ref{tab:Gfour} lists 24 more forts.
A string such as \(0125\) denotes \(\set{0,1,2,5}\).

\begin{table}[H]
\centering
\small
\caption{The four-vertex fort certificate for \(G\).}
\label{tab:Gfour}
\begin{tabularx}{\textwidth}{@{}cY@{}}
\toprule
\(\sbd GF\) & Four-vertex forts \(F\) \\
\midrule
\(\set{8}\) & \(0125,0126,0136,0156,0234,0236,0245,0246,0256,0346,0456\) \\
\(\varnothing\) & \(1256,1257,1267,1567,2346,2347,2367,2456,2457,2467,2567,3467,4567\) \\
\bottomrule
\end{tabularx}
\end{table}

Suppose a five-set \(B\) were 1-leaky PSD forcing, and put \(Q=V(G)\setminus B\).
The fort \(\set{8}\) forces \(8\in B\), so \(Q\subseteq\set{0,\ldots,7}\) and \(|Q|=4\).
The three two-vertex forts imply that \(Q\) contains none of
\[
  \set{1,4},\qquad \set{3,5},\qquad \set{0,7}.
\]
The two three-vertex forts imply that \(Q\) contains neither \(\set{1,2,3}\) nor \(\set{1,3,7}\).
The 24 entries in Table~\ref{tab:Gfour} are exactly the four-subsets of \(\set{0,\ldots,7}\) with these five avoidance properties.
Thus \(Q\) is itself a fort.
It is disjoint from \(B\), contrary to Theorem~\ref{thm:fort}.
No five-set works.
Extend any smaller set to five vertices; the same argument gives a disjoint fort.
Thus \(\zl G\ge6\).
The forcing table gives equality.
\end{proof}

\begin{proof}[Proof of Theorem~\ref{thm:main}]
The universal inequalities are Corollary~\ref{cor:edge}.
Theorem~\ref{thm:uniform} gives the negative extreme for every \(\ell\ge1\).
Theorem~\ref{thm:positive} gives the positive extreme for every \(\ell\ge2\).
Propositions~\ref{prop:H} and \ref{prop:G} give the positive extreme
\[
  \zl G-\zl{G-e}=6-4=2.
\]
Both graphs are connected.
The order is least possible for connected pairs by the verification through order eight in \cite{Elias2025}.
\end{proof}

\section{Computational verification}

The supplement contains one standalone Python script.
It uses only the standard library.
It performs five exact audits.

First, the script searches the PSD color-change state graph for each initial set and each leak.
Second, it enumerates every nonempty connected vertex set \(F\), keeps those with \(|\sbd XF|\le1\), and finds all minimum transversals.
The two methods agree:
\[
  \zl H=4,\qquad \zl G=6.
\]
The enumeration finds 194 connected forts in \(H\), 218 in \(G\), two minimum sets for \(H\), and 36 for \(G\).
It also checks every boundary and every forcing sequence used above.

Third, the script stress-tests Theorem~\ref{thm:support} for both color-change rules.
For all labeled graphs through order four, it checks every support \(S\), every admissible graph pair, every initial blue set, and every leak set.
This gives 516,416 standard and 596,184 PSD transfer implications.
For a single edited edge, the audit extends through order five and checks 2,683,436 standard and 3,297,296 PSD implications.
It also verifies 63,462 PSD parameter comparisons over all leak levels.
The same enumeration checks the low-degree obstruction for both rules through order five.
It verifies 6,504 instances of the connected-fort transversal characterization.

Fourth, the script tests Theorem~\ref{thm:uniform} for \(1\le\ell\le8\).
It examines every leak placement under both rules, for 118,556 fixed-leak instances.

Finally, it tests Theorem~\ref{thm:positive} for \(2\le\ell\le9\).
It checks all 791,374 fixed-leak upper-bound instances and 1,012 lower-bound fort certificates.

\section{Conclusion}

The conjectured increase of at most one is false, but the optimal universal replacement is two.
The bound holds for every leak level, in both edge-edit directions, and for both PSD and standard leaky forcing.
Its standard specialization is Theorem~5.2 of Bjorkman et al. \cite{Bjorkman2025}; its PSD specialization is the sharp replacement for \eqref{eq:conjecture}.
For PSD forcing, both signs are attained at every positive leak level.
The two-clique bridge family gives the negative equality, while the connected clique-leaf family gives the positive equality for every \(\ell\ge2\).
The blue-support theorem also controls simultaneous edits supported on a fixed vertex set.
Its endpoint profile recovers the conjectured bound whenever a minimum deleted-edge set contains an endpoint.
The connected-fort theorem reduces the obstruction family of Elias et al. to an explicit singleton-boundary hypergraph.
The order-nine example supplies the remaining one-leak equality case.
Appendix~\ref{app:repairs} records why several stronger local repair principles fail.

\appendix
\section{Failed local repair principles}\label{app:repairs}

The endpoint profile in Corollary~\ref{cor:endpoint} explains the extremal jump.
Exact fort enumeration gives precisely two minimum fort transversals of \(H\):
\[
  B_5=\set{2,3,5,8},\qquad B_6=\set{2,3,6,8}.
\]
Both miss the disjoint \(G\)-forts \(\set{1,4}\) and \(\set{0,7}\).
Hence no single added vertex repairs either set.
In particular, \(r_H(0,1)=0\), and the support-transfer bound is attained.

Adding both endpoints and deleting one old vertex also fails.
Table~\ref{tab:exchange} gives a disjoint fort for every possible deletion.

\begin{table}[H]
\centering
\caption{Obstructions to the two-endpoint, one-out exchange.}
\label{tab:exchange}
\begin{tabular}{@{}ccc@{}}
\toprule
Initial set & Deleted vertex \(b\) & Fort disjoint from \((B_i-b)\cup\set{0,1}\) \\
\midrule
\(B_5\) & \(2,3,5,8\) & \(2467,3467,4567,\set{8}\), respectively \\
\(B_6\) & \(2,3,6,8\) & \(2457,\set{3,5},4567,\set{8}\), respectively \\
\bottomrule
\end{tabular}
\end{table}

A natural four-corner descent principle is also false.
For \(G=H+uv\), define
\[
  \Phi(B)=\sum_{b\in B}\bigl(d_H(b,u)+d_H(b,v)\bigr).
\]
The proposed principle says the following.
If no \(B+x\) is a \(G\)-transversal, and no
\((B-b)\cup\set{u,v}\) is a \(G\)-transversal, then an exchange
\(B-b+x\) remains a minimum \(H\)-transversal and lowers \(\Phi\).

In the present graph, \(u=0\) and \(v=1\).
One has \(\Phi(B_6)=13<14=\Phi(B_5)\).
Thus \(B_6\) is the unique distance-minimizing transversal.
Its only exchange to another minimum transversal replaces \(6\) by \(5\), which raises \(\Phi\).
Let \(\mathcal F_0\) be the \(H\)-forts disjoint from \(\set{0,1}\).
Exact enumeration gives \(\tau(\mathcal F_0)=4=\zl H\).
Thus the strengthened endpoint-avoiding hypothesis does not restore the descent claim.

Even the basic endpoint-repair assertion fails on five vertices.
Let \(H_5\) be the path \(4-0-2-1-3\), and let \(G_5=H_5+01\).
The unique minimum \(H_5\)-transversal is \(\set{3,4}\).
Neither endpoint repairs it, although vertex \(2\) does.
The supplementary verifier checks these statements exactly.

\begin{samepage}
\paragraph{Data and Code Availability.}
The exact verifier is included in the Supplementary Materials accompanying this record.

\end{samepage}

\end{document}